\documentclass[11pt,reqno]{amsart} 
\usepackage[left=2.9cm,right=2.9cm]{geometry}
\usepackage{amsfonts}
\usepackage{amsmath}
\usepackage{amssymb}
\usepackage{mathtools}
\usepackage{amsthm}
\usepackage{pgfplots}
\pgfplotsset{compat=newest}
\usepackage[shortlabels]{enumitem}
\usepackage{tikz}   
\usepackage{tikz-cd}  
\usepackage{changepage}
\usetikzlibrary{arrows}
\usepackage{bm}
\usepackage{graphicx}
\usepackage{cases}
\usepackage{appendix}
\usepackage{esint}
\usepackage{hyperref}
\usepackage{mathrsfs}
\usepackage{mathrsfs}
\numberwithin{equation}{section}

\theoremstyle{definition}
\newtheorem{thm}{Theorem}

\newtheorem{defn}[thm]{Definition}
\newtheorem{lem}[thm]{Lemma}
\newtheorem{prop}[thm]{Proposition}

\numberwithin{thm}{section}

\newcommand{\R}{\mathbb{R}}  
\newcommand{\Sp}{\mathbb{S}}
\newcommand{\p}{\partial}  

\newcommand{\dif}{\textup{d}} 

\newcommand{\Hau}{\mathcal{H}} 

\newcommand{\sing}{\operatorname{sing}}

\usepackage{xcolor}
\usepackage{listings}
\begin{document}
\title[Regularity of stable harmonic maps to spheres]{Optimal regularity of stable harmonic maps to spheres}
\author{Xuanyu Li}
\address{Department of Mathematics, Cornell University, Ithaca, NY 14853, USA}
\email{xl896@cornell.edu}
\begin{abstract}
    In this paper, we show that the codimension of the singular set of a stable stationary harmonic map to a round $k$-sphere is at least $k+1$ when $k$ is between 3 and 6, and is at least 7 when $k$ is at least 7. The result is sharp in the sense that there exist energy minimizing 0-homogeneous maps in the aforementioned critical dimensions. We also establish non-trivial index bounds for the non-constant harmonic maps from $n$-spheres to $k$-spheres, provided that $n$ is less than $k$.
\end{abstract}
\maketitle

\section{Introduction}\label{s: intro}
\subsection{Main results}
\textbf{Harmonic map}s are critical points of the Dirichlet energy for maps between Riemannian manifolds. A smooth map $u$ from a Riemannian manifold $M$ to the round sphere $\Sp^k$ is harmonic if and only if it satisfies the equation
$$\Delta_Mu^i+\vert\nabla_Mu\vert^2u^i=0\text{ for all }i=1,\dots,k+1.$$
Throughout this paper, we identify $\Sp^k$ with the subset $\lbrace p\in\R^{k+1}:\vert p\vert=1\rbrace$ in $\R^{k+1}$. Accordingly, we write $u=(u^1,\dots,u^{k+1})$ and the harmonic map equation above is interpreted componentwise. A map in $W^{1,2}(M,\Sp^k)$ is called a \textbf{weak} harmonic map if it satisfies this equation in the distributional sense.

Weak harmonic maps need not be continuous. For general target manifolds, especially manifolds with positive curvature, singularities may exist even under substantially stronger variational hypotheses, such as energy minimizing or stable stationary. A fundamental example is the radial projection $x\mapsto x/\vert x\vert,\R^{k+1}\rightarrow\Sp^k$, which minimizes the Dirichlet energy among maps with the same boundary trace for all $k\geqslant2$; as shown by W. J\"ager and H. Kaul \cite{JagerKaulMinimizing}, H. Brezis, J. Coron and E. Lieb \cite{BrezisCoronLieb}, F. Lin \cite{LinMinimizing} and J. Coron, R. Gulliver \cite{CoronGulliver}. Without stationarity, a weak harmonic map to sphere can even be everywhere discontinuous as shown by an example constructed by T. Riviere \cite{RiviereDiscontinuous}.

The regularity theory for stable stationary harmonic maps to spheres was developed by R. Schoen and K. Uhlenbeck \cite{SchoenUhlenbeckSphere} and F. Lin and C. Wang \cite{LinWangSphere}. Their work establishes strong partial regularity results, but leaves open some dimensions. More precisely, for map from a six-dimensional domain into $\Sp^k$ with $6\leqslant k\leqslant 9$, the existing theory does not determine whether singularities can occur: it neither implies full regularity nor provides singular examples. Thus, the unresolved issue is qualitative rather than merely quantitative--it remained unclear whether the gap reflected genuine singular behavior or a limitation of the available estimates.

Resolving this borderline case is important both for the intrinsic regularity theory and for geometric applications of harmonic maps. For example, the regularity theory above is an essential input in the partial regularity of harmonic maps constructed by M. Karpukhin and D. Stern \cite{MikhailStern}. Sharper regularity for stable stationary maps therefore yields correspondingly stronger structural information about these variationally constructed maps. 

In this paper, we fill the gaps in the previous works by establishing the following sharp result. In particular, we improve the singular set dimension bound in the work of F. Lin and C. Wang for $6\leqslant k\leqslant 9$.
\begin{thm}\label{thm: main theorem 1}
    Let $M$ be a Riemannian $n$-manifold and $u: M\rightarrow \Sp^k$ be a stable stationary harmonic map. Then $u$ is smooth away from a relatively closed set $\sing u$, whose Hausdorff dimension is bounded by
    \begin{align*}
        \dim_{\Hau} \sing u\leqslant
        \begin{cases}
            n-k-1,&\text{ if }3\leqslant k\leqslant6;\\
            n-7,&\text{ if }k\geqslant 7.
        \end{cases}
    \end{align*}
\end{thm}
We refer readers to Section \ref{s: pre} for the precise definition of stability and stationarity of a harmonic map.

For $3\leqslant k\leqslant 6$, the singular model of the critical dimensions in Theorem \ref{thm: main theorem 1} is given by $x\mapsto x/\vert x\vert, \R^{k+1}\rightarrow \Sp^k$ while the singular map $x\mapsto (x/\vert x\vert,0,\dots,0),\R^7\rightarrow\Sp^k$ is shown to be minimizing when $k\geqslant7$ by R. Schoen and K. Uhlenbeck in \cite{SchoenUhlenbeckSphere}. Given the examples here, the regularity result above is sharp.

When $k\geqslant 3$, it is well-known that $\Sp^k$ does not support any stable harmonic 2-sphere. Hence, the analysis by D. Hsu \cite{HsuStableharmonicmap} (see also D. Hsu and J. Li\cite{HsuLiregularity}) and the dimension reduction technique by H. Federer \cite{Federer} readily apply (see also F. J. Almgren \cite{Almgren}), which reduces the proof of Theorem \ref{thm: main theorem 1} to the following result.
\begin{thm}\label{thm: main theorem 2}
    Let $u: \R^{n+1}\rightarrow\Sp^k$ be a regular 0-homogeneous stable stationary harmonic map. Suppose 
    \begin{align*}
        n+1\leqslant\begin{cases}
            k,&\text{ if }3\leqslant k\leqslant 6;\\
            6,&\text{ if }k\geqslant7,
        \end{cases}
    \end{align*}
    then $u$ must be a constant.
\end{thm}
Here we say a map $u$ defined on $\R^{n+1}$ is a regular 0-homogeneous map, provided that $u\in C^{\infty}(\R^n\setminus\{0\})$ and $u(\lambda x)=u(x)$ for all $\lambda>0$.

Finally, when a nonconstant stable tangent map exists, we also deduce a non-trivial index lower bound.

\begin{thm}\label{thm: main theorem 3}
    Let $u: \Sp^n\rightarrow\Sp^k$ be a non-constant harmonic map. Suppose $6\leqslant n\leqslant k-1$, then $u$ has index at least $n+2$ with respect to the energy functional.
\end{thm}

Recall that Y. L. Xin \cite{Xin} and A. El Soufi \cite{ElSoufiIndex} show that the index of a non-constant harmonic map $u\in C^{\infty}(\Sp^n,\Sp^k)$ is at least $n+1$, with the eigensections given by the covariant derivatives of $u$ with respect to conformal vector fields. We also refer readers to Section \ref{ss: pre, jacobi operator and index} below for details on the definition of index and Xin and El Soufi's index estimate. Hence, Theorem \ref{thm: main theorem 3} is non-trivial in that it provides at least one additional negative direction for the second variation of the map in the theorem.

\subsection{History and related results}

One of the central problems in the study of harmonic maps is its regularity theory. In their pioneering work, J. Eells and J. H. Sampson \cite{EellsSampson} initiated the study of harmonic maps, where they proved that the homotopy groups of between compact manifolds $M$ and manifold $N$ with nonpositive sectional curvature can be represented by smooth harmonic maps. S. Hildebrandt, H. Kaul and K. Widman \cite{HildebrandtKaulWidman} further establish a full regularity theory when the target manifold is negatively curved, where all weak harmonic maps are shown to be smooth. They also establish smoothness under the assumption that the image of the map is contained in a sufficiently small geodesic ball. 

The regularity of harmonic maps from 2-dimensional domains is also well-understood. Smoothness was established successively for energy minimizing maps by C. B. Morrey \cite{Morrey}, for stationary harmonic maps by R. Schoen \cite{SchoenAnalyticAspects} and finally for arbitrary weak harmonic maps by F. H\'elein \cite{Helein}.

The partial regularity theory for harmonic maps from higher dimensional domains to general targets originated in the seminal work of R. Schoen and K. Uhlenbeck \cite{SchoenUhlenbeckRegularity}. They establish the monotonicity formula and $\epsilon$-regularity theorem of energy minimizing maps. Combining these with the dimension reduction principle of H. Federer \cite{Federer}, they prove that energy minimizing maps are smooth away from a set of codimension at least 3. In \cite{SchoenUhlenbeckSphere}, they also show that the energy minimizing maps to hemispheres have singular sets of codimension at least 7.

However, the regularity for stationary harmonic maps is considerably more delicate. Even the proof of $\epsilon$-regularity is substantially more difficult than in the energy minimizing setting: it is established by L. C. Evans \cite{EvansSphere} for sphere valued maps and by F. Bethuel \cite{Bethuel} for general targets. T. Rivi\`ere and M. Struwe \cite{RiviereStruwe} extend $\epsilon$-regularity to $C^2$ target manifold using the method introduced by T. Rivi\`ere \cite{RiviereConservationLaw}. A further difficulty is that sequences of stationary harmonic maps need not converge strongly in $W^{1,2}$, even subsequentially.  The resulting loss of compactness gives rise to the defect measures and complicates dimension reduction argument. Consequently, currently it is only known that the singular set of a general stationary harmonic map has zero codimension 2 Hausdorff measure and a longstanding open question ask that whether the singular set must in fact have codimension at least 3. F. Lin \cite{LinGradientEstimates} gives a general description of this loss of compactness in terms of defect measures in arbitrary dimensions. The defect of energy in the convergence of harmonic maps is fully captured by bubbles, namely harmonic maps
from 2-sphere to the target. This phenomenon is known as the energy identity. It is established by J. Jost \cite{Jost}, T. H. Parker \cite{ParkerBubbleTree} for two dimensional domains, F. Lin and T. Rivi\`ere \cite{LinRiviere} for sphere-valued maps and finally by A. Naber and D. Valtorta \cite{naber2024energyidentitystationaryharmonic} for general targets.

In some cases the bubbling is avoidable. For example, in the stable stationary case, D. Hsu \cite{HsuStableharmonicmap} observes that the energy defect arising in the convergence of stable stationary harmonic maps is accounted for by stable harmonic 2-spheres; see also D. Hsu and J. Li\cite{HsuLiregularity} and M. Karpukhin and D. Stern \cite{MikhailStern}. Consequently, if the target manifold admits no stable harmonic 2-sphere, the dimension reduction argument can be applied once more to analyze the singular sets. The absence of stable harmonic 2-sphere can be guaranteed if the target has dimension at least 3 and has positive sectional curvature; see A. El Soufi \cite{ElSoufiTwoSphere}. In this case, the singular set of a stable stationary harmonic map has codimension at least 3. For homogeneous target spaces, results of T. Okayasu \cite{Okayasu} and J. Krantz \cite{Krantz} improves this bound to codimension at least four.

Finally, we list some works regarding the harmonic maps to spheres. The theory for this specific target is more accessible because of its high symmetry and constant curvature. W. J\"ager and H. Kaul \cite{JagerKaulMinimizing} study the stability and minimizing properties of the radial projection maps while T. Nakajima \cite{NakajimaRigidity,NakajimaIndex} establishes the rigidity of those maps under degree or index assumptions. E. Calabi \cite{Calabi}, J. Barbosa \cite{Barbosa}, M. Karpukhin \cite{KarpukhinIndex} obtain refined index estimates for harmonic 2-spheres in $n$-spheres.

\subsection{Sketch of the proofs}
We briefly sketch the main ideas behind the proofs of Theorem \ref{thm: main theorem 2} and Theorem \ref{thm: main theorem 3}. In earlier works, including R. Schoen and K. Uhlenbeck \cite{SchoenUhlenbeckSphere}, F. Lin  and C. Wang \cite{LinWangSphere}, weaker results in the same direction of Theorem \ref{thm: main theorem 2} are typically obtained by inserting the test fields $\vert\nabla u\vert\hat{e}_{j}\circ u$ in the second variation form $I_u$, summing over all $j$, and applying a Kato-type inequality for $\vert\nabla^2 u\vert$, Here $\hat e_{j}$ is the conformal vector field of $\Sp^k$. 

Although effective, this argument is generally not sharp when $n<k$. Indeed, summing over $j$ assigns the same weight to every ambient direction and therefore produces an isotropic estimate. Such an averaging is well adapted to the fully symmetric model $x\mapsto x/\vert x\vert$ when $n=k$, but not the case for the equatorial map $x\mapsto(x/\vert x\vert,0,\dots,0)$, when $n<k$. The advantage of choosing a geometrically distinguished direction is already visible in R. Schoen and K. Uhlenbeck's derivation of the optimal regularity estimate for minimizing harmonic maps into a  hemisphere \cite{SchoenUhlenbeckSphere}. A related phenomenon appears in J. Simons' analysis of stable codimension one minimal cones \cite{SimonsMinimal}, where the normal direction provide much more effective test variations than tangential directions.

Our approach is designed to preserve this directional information. Rather than averaging over all target directions, we choose test directions adapted to the geometry of the map and exploit those that contribute most strongly in the estimates. This yields sharper estimates while avoiding the more involved calculations required to derive the Kato-type inequality. 

In order to prove Theorem \ref{thm: main theorem 2}, we restrict the map to $\Sp^n$. The major type of the vector field we use is $\hat e_j\circ u-(e_j\wedge \bar u)u$, where $\bar u$ is the average of $u$ over domain $\Sp^n$, $(e_j\wedge\bar u)v=\langle\bar u,v\rangle e_j-\langle e_j,v\rangle \bar u$ is a skew-symmetric matrix, and $\lbrace e_j\rbrace$ is an orthonormal basis of a $d$-subspace $L\subset\R^{k+1}$. We calculate in Proposition \ref{prop: control wrt L} that this gives a bound 
$$3\int_{\Sp^n}\vert\nabla u\vert^2\vert\Pi_L u\vert^2\dif\sigma-dE(u)\geqslant\frac{(n-1)^2}{4}\left(\int_{\Sp^n}\vert\Pi_{L}(u-\bar u)\vert^2\dif \sigma-d(1-\vert\bar u\vert^2)\right).$$
Selecting $L$ to be the full space $\R^{k+1}$, we get
$$E(u)\leqslant\frac{(n-1)^2k}{4(k-2)}(1-\vert\bar u\vert^2).$$

On the other hand, it is important to bound the energy $E(u)$ from below. This part makes use of the spherical harmonic functions, which are the eigenfunctions of Laplace operators on $\Sp^n$. We expand the harmonic map of interest with respect to these eigenfunctions:
$$u(x)=\bar u+\mathscr{H}_u x+\text{ higher order terms},x\in\Sp^n.$$
In particular, this gives rise to a $(k+1)\times(n+1)$ matrix $\mathscr{H}_u$. When $k>n$, $(\operatorname{Im}\mathscr{H}_u)^{\perp}$ is non-trivial. If $L\subset (\operatorname{Im}\mathscr{H}_u)^{\perp}$, then $\Pi_L$ does not have first order term in spherical harmonic function decomposition. As such, we have improved constants in the Poincare inequality, since the third eigenvalue of $\Sp^n$ is $(2n+1)$:
$$\int_{\Sp^n}\vert\nabla\Pi_L u\vert^2\dif\sigma\geqslant2(n+1)\int_{\Sp^n}\vert\Pi_L(u-\bar u)\vert^2\dif\sigma.$$
Of course, the $L^{\perp}$ direction is controlled by normal Poincare inequality
$$\int_{\Sp^n}\vert\nabla\Pi_{L^{\perp}}u\vert^2\dif\sigma\geqslant n\int_{\Sp^n}\vert\Pi_{L^{\perp}}(u-\bar u)\vert^2\dif\sigma.$$
Combining with the first inequality above, we get lower bound of the energy
$$\left(\frac{(6n+6)n}{n+2}-\frac{(n-1)^2}{4}\left(d+\frac{n}{n+2}\right)\right)(1-\vert\bar u\vert^2)\leqslant\left(\frac{-n^2+14n-1}{4(n+2)}+(3-d)\right)E(u).$$

By selecting $d=1$ when $2\leqslant n\leqslant 4,k\geqslant n+1$ and $d=2$ when $n=5,k\geqslant n+2$, we can explicitly calculate that the coefficient of the upper bound of $E(u)$ derived above is less than that of the lower bound, forcing $E(u)=0$. For the index bound in Theorem \ref{thm: main theorem 3}, since the eigenvalue $-(n-2)$ occurs with multiplicity at least $n+1$, if we assume the index is exactly $n+1$, the first eigenvalue of Jacobi operator of $u$ is bounded below by $-(n-2)$. Similar reasoning, replacing $(n-1)^2/4$ by $(n-2)$ above, shows that such a map must be a constant. As mentioned at the beginning of this subsection, we should regard the directions in this $d$-subspace $L$ as the dominant directions of the test vector fields.

However, the case $(n,k)=(5,6)$ is more involved. Assume by contradiction $u$ is not a constant, the argument above only gives $(1-\vert \bar u\vert^2)/E(u)=1/6+\delta$ where $\delta\in[0,1/44]$. In order to improve the estimates, we carefully analyze the eigenvalues of $\mathscr{H}_u\mathscr{H}_u^{\text{T}}$. They are square of singular values of $\mathscr{H}_u$. We first take $L$ equals to the space corresponding to first $d$ eigenvectors of $\mathscr{H}_u\mathscr{H}_u^{\text{T}}$. Using weak supermajorization of vectors, we are able to derive lower bound 

$$\left(\sum_{j=1}^7\sqrt{\lambda_j}\right)^2\geqslant\frac{600}{143}E(u).$$

On the other hand, by taking $X_{\alpha}=\nabla_{\hat e_{\alpha}}u,Y_j=\hat e_j\circ u-(e_j\wedge \bar u)u$ in second variation of $u$ respectively, noticing that $(I_u(Y_j,X_{\alpha})+4\langle X_{\alpha},Y_j\rangle_{L^2})$ equals to $5\mathscr{H}_u$/6, we are able to use the stability inequality of $u$ to derive the lower bound

$$\left(\sum_{j=1}^7\sqrt{\lambda_j}\right)^2\leqslant\frac{144}{25}E(u)(6(1-\vert\bar u\vert^2)-E(u)).$$

Combining above two inequalities, the resulting upper bound exceeds the resulting upper bound, leading to a contradiction. Hence the remaining part of Theorem \ref{thm: main theorem 2} is also proved.

\subsection{Organization of the paper}
In Section \ref{s: pre}, we introduce necessary notations and definitions, as well as establish some preliminary results. In particular, in Section \ref{ss: pre, stable maps}, we introduce the stability inequality; In Section \ref{ss: pre, jacobi operator and index}, we introduce the Jacobi operator and index for harmonic maps; In Section \ref{ss: spherical harmonic}, we introduce spherical harmonic functions, and use it to define the spherical harmonic component of a map. In Section \ref{s: part 1}, we prove Theorem \ref{thm: main theorem 2} for $(n,k)\ne(5,6)$. In Section \ref{s: part 2}, we solve the remain $(n,k)=(5,6)$ case for Theorem \ref{thm: main theorem 2}. Finally, in Section \ref{s: index estimate}, we prove the index estimate in Theorem \ref{thm: main theorem 3}.

\subsection*{Acknowledgement}

The author is highly grateful to Prof. Xin Zhou and Prof. Daniel Stern for their valuable discussions on related topics, guidance and constant support. The author also wants to thank Zhihan Wang for fruitful discussions and his interest in this work. The author is supported by NSF grant DMS-2404992.

\subsection*{AI usage statement}

This work is assisted by ChatGPT 5.6 Sol. In particular, the proofs of Lemma \ref{lem: eigenvalue lower bound} and Lemma \ref{lem: eigenvalue upper bound} are given by ChatGPT while other proofs are due to the author. The author also makes necessary simplification to the proof given by AI. When preparing this manuscript, all the context and mathematical proof are entirely written by the author, and the AI is only used to improve its exposition. The author takes full responsibility for the final content of the paper.

\section{Preliminaries}\label{s: pre}
We first recall some preliminary results regarding harmonic maps. For simplicity, let us restrict the discussion to harmonic maps from a domain $\Omega$ in $\R^n$ or from round sphere $\Sp^n$. Denote by $\langle\cdot,\cdot\rangle$ the standard metric on $\R^{k+1}$. Note that it induces the round metric on $\Sp^k$. We refer readers to F. Lin and C. Wang \cite{LinWangBook} for a more comprehensive understanding of the harmonic maps.

\subsection{Stable stationary harmonic maps to spheres}\label{ss: pre, stable maps}

\begin{defn}\label{defn: definition of stationary and stable}
    Let $u\in W^{1,2}(\Omega,\Sp^k)$ be a weakly harmonic map. We say $u$ is 
    \begin{enumerate}
        \item \textbf{stationary}, if $u$ satisfies
    $$\int_{\Omega}\vert\nabla u\vert^2\operatorname{div}\xi\dif x=2\sum_{\alpha,\beta=1}^n\int\langle\p_{\alpha}u,\p_{\beta}u\rangle\p_{\alpha}\xi^{\beta},\text{ for all }\xi\in C_c^{\infty}(\Omega,\R^n);$$
        \item \textbf{stable}, if $u$ satisfies 
        $$\frac{\dif^2}{\dif t^2}\bigg|_{t=0}\int_{\Omega}\left\vert\nabla\frac{u+t X}{\vert u+t X\vert}\right\vert^2\dif x\geqslant0,\text{ for all }X\in W^{1,2}_c\cap L^{\infty}(\Omega,\R^{k+1}),\langle X,u\rangle=0\text{ a.e.}$$
        Or equivalently,
        $$\int_{\Omega}\vert(\nabla X\vert^2-\vert\nabla u\vert^2\vert X\vert^2)\dif x\geqslant0,$$
        for such $X$.
    \end{enumerate}
\end{defn}
For a vector field $X$ which takes values in $\R^{k+1}$ and a map $u$ to $\Sp^{k}$, the condition $\langle X(x),u(x)\rangle=0$ is equivalent to $X(x)\in T_{u(x)}\Sp^k$. If such a condition holds for almost every $x$ in the domain, we say that $X$ takes value in $T_u\Sp^k$. For simplicity, we denote the space of $X$ in Definition \ref{defn: definition of stationary and stable} (2) by $X\in W^{1,2}_c\cap L^{\infty}(\Omega,T_u\Sp^k)$. Other spaces of vector fields taking value in $T_u\Sp^k$ are similarly defined.

\begin{defn}
    There are two types of vector fields that will be of particular interest in this paper.
    \begin{enumerate}
        \item Given a unit vector $e\in\R^{k+1}$, define the conformal vector field on $\Sp^k$ by 
        $$\hat{e}(p)=e-\langle e,p\rangle p,p\in\Sp^k.$$
        They are the infinitesimal generators of the conformal diffeomorphisms of $\Sp^k$.
        \item Given a $(k+1)\times(k+1)$ skew-symmetric matrix $A\in\mathfrak{o}(k+1)$, it gives rise to a Killing vector field on $\Sp^k$ by 
        $$Ap,p\in\Sp^k.$$
        Recall that $\mathfrak{o}(k+1)$ is the Lie algebra of $O(k+1)$. In other words, $Ap$ is an infinitesimal generator of the isometry of $\Sp^k$.
    \end{enumerate}
    Given two vectors $v,w$ in $\R^{k+1}$, the skew-symmetric matrix induced by them is
    $$(v\wedge w)p=\langle w,p\rangle v-\langle v,p\rangle w.$$
\end{defn}

In general, the second variation of a harmonic map to a sphere is given by the following bilinear form.
\begin{defn}
    Let $u\in W^{1,2}(\Omega,\Sp^k)$ be a harmonic map. Given $X,Y\in W^{1,2}\cap L^{\infty}(\Omega,T_u\Sp^k)$. We define
    $$I_u(X,Y)=\frac{\p^2}{\p s\p_t}\bigg|_{s,t=0}\frac{1}{2}\int_{\Omega}\left\vert\nabla u_{s,t}\right\vert^2=\int_{\Omega}\langle\nabla X,\nabla Y\rangle-\vert\nabla u\vert^2\langle X,Y\rangle,$$
    where $(s,t)\mapsto u_{s,t}$ is a family of maps in $W^{1,2}(\Omega,\Sp^k)$ defined in a neighborhood of $(0,0)$ such that $\p_su_{s,t}|_{s,t=0}=X,\p_tu_{s,t}|_{s,t=0}=Y$.
\end{defn}

Since target isometries preserves the energy, we have the following consequence.

\begin{prop}\label{prop: infinitesimal isometry}
    Let $u\in W^{1,2}(\Omega,\Sp^k)$ be a harmonic map. For any $A\in\mathfrak{o}(k+1)$ and $X\in W^{1,2}_c\cap L^{\infty}(\Omega,T_u\Sp^k)$, we have that
    $$I_u(Au,X)=0.$$
\end{prop}
\begin{proof}
    Let $\phi_s:\Sp^k\rightarrow\Sp^k,s\in(-\epsilon,\epsilon)$ be the follow generated by the Killing vector field $p\mapsto Ap,p\in\Sp^k$. They are isometries of $\Sp^k$. Hence $E(\phi_s\circ u_t)=E(u_t)$ for all $s\in(-\epsilon,\epsilon)$, where $u_t=(u+t X)/\vert u+t X\vert$. As such $\p_s E(\phi_s\circ u_t)=0$ for any $s$, and 
    $$0=\frac{\p}{\p_t}\frac{\p}{\p_s}\bigg|_{s,t=0}\frac{1}{2}\int_{\Omega}\vert\nabla \phi_s\circ u_t\vert^2=I_u(Au,X).$$
\end{proof}

The regularity theory by R. Schoen and K. Uhlenbeck \cite{SchoenUhlenbeckRegularity} and D. Hsu \cite{HsuStableharmonicmap} tells us that singularities of stable stationary harmonic maps are modeled on tangent maps. Tangent maps are \textbf{0-homogeneous maps}, namely, maps $\tilde u:\R^n\rightarrow\Sp^k$ which are scaling invariant $\tilde u(\lambda x)=\tilde u(x)$ for all $\lambda>0$. We say a 0-homogeneous map $\tilde u$ is \textbf{regular}, if $\tilde u\in C^{\infty}(\R^n\setminus\{0\},\Sp^k)$. The study of 0-homogeneous maps typically reduces to the study of their restriction to $\Sp^{n-1}$. Note that the aforementioned $\tilde u$ is a harmonic map from $\R^n$ to $\Sp^k$ if and only if $u=\tilde u|_{\Sp^{n-1}}$ is a harmonic map from $\Sp^{n-1}$ to $\Sp^k$.

The following relationship of the second variation between a 0-homogeneous map and its restriction to sphere is also frequently used.
\begin{prop}[R. Schoen and K. Uhlenbeck \cite{SchoenUhlenbeckSphere}]\label{prop: stable inequality}
    Let $u\in C^{\infty}(\Sp^n,\Sp^k)$ be a harmonic map. If its homogeneous extension $\tilde u(x)=u(x/\vert x\vert)$ is stable, then
    $$I_u(X,X)\geqslant-\frac{(n-1)^2}{4}\int_{\Sp^n}\vert X\vert^2\dif\sigma,$$
    for all $X\in C^{\infty}(\Sp^n,T_u\Sp^k)$.
\end{prop}

We need the following consequence of the stable inequality above. For simplicity, let us denote the normalized volume measure of $\Sp^n$ by $\sigma$, which is a constant multiple of the standard volume measure on $\Sp^n$ such that $\sigma(\Sp^n)=1$. All the integrals on $\Sp^n$ in this paper are taken with respect to $\sigma$. In particular, set the energy of $u$ to be
$$E(u):=\int_{\Sp^n}\vert\nabla u\vert^2\dif\sigma.$$

\begin{prop}\label{prop: control wrt L}
    Let $u$ be the harmonic map in Proposition \ref{prop: stable inequality}. For any $d$-dimensional subspace $L$ of $\R^{k+1}$, denoting by $\Pi_L$ the orthonormal projection to $L$, we have that
    $$3\int_{\Sp^n}\vert\nabla u\vert^2\vert\Pi_L u\vert^2\dif\sigma-dE(u)\geqslant\frac{(n-1)^2}{4}\left(\int_{\Sp^n}\vert\Pi_{L}(u-\bar u)\vert^2\dif \sigma-d(1-\vert\bar u\vert^2)\right),$$
    where $\bar u=\int_{\Sp^n}u\dif\sigma\in\R^{k+1}$ denotes the average of $u$ over $\Sp^n$.
    In particular, by taking $L=\R^{k+1}$, we get that
    $$(k-2)E(u)\leqslant\frac{(n-1)^2k}{4}(1-\vert\bar u\vert^2).$$
\end{prop}
\begin{proof}
    The last inequality of Proposition \ref{prop: control wrt L} comes from the fact that $\int_{\Sp^n}\vert u-\bar u\vert^2\dif\sigma=1-\vert\bar u\vert^2$. We only need to show the first inequality of Proposition \ref{prop: control wrt L}. To this end, let $\{e_1,\dots,e_{d}\}$ be an orthonormal basis of $L$. Let us consider the vector field $$X_j=\hat e_j\circ u-(e_j\wedge \bar u)u.$$
    Taking $X_j$ in the stability inequality and using Proposition \ref{prop: infinitesimal isometry} gives
    \begin{equation}\label{eq: calculation of second variation on conformal vector fields}
    \begin{aligned}
        I_u(X_j,X_j)&=I_u(\hat e_j\circ u,\hat e_j\circ u)\\&=\int_{\Sp^n}(\vert\nabla \hat e_j\circ u\vert^2-\vert\nabla u\vert^2\vert\hat e_j\circ u\vert^2)\dif\sigma\\&=\int_{\Sp^n}(\vert\nabla\langle e_j,u\rangle\vert^2+\langle e_j,u\rangle^2\vert\nabla u\vert^2-\vert\nabla u\vert^2(1-\langle e_j,u\rangle^2))\dif \sigma\\&=\int_{\Sp^n}3\langle e_j,u\rangle^2\vert\nabla u\vert^2\dif\sigma-E(u).
    \end{aligned}
    \end{equation}
    Note that, when computing $\vert\nabla \hat e_j\circ u\vert^2$, we used the fact that $\langle u,\p_{\alpha} u\rangle=0$ for all $\alpha$, which is a consequence of $\vert u\vert=1$. The last equality of \eqref{eq: calculation of second variation on conformal vector fields} is computed by multiplying the harmonic map equation $\Delta \langle u,e_j\rangle=-\vert\nabla u\vert^2\langle u,e_j\rangle$ by $\langle u,e_j\rangle$ and integration by parts, which gives
    \begin{equation}\label{eq: energy equality of components}
        \int_{\Sp^n}\vert\nabla u\vert^2\langle u,e_j\rangle^2\dif\sigma=\int_{\Sp^n}\vert\nabla\langle u,e_j\rangle\vert^2\dif\sigma.
    \end{equation}
    On the other hand, since $$\vert(e_j\wedge \bar u)u\vert^2\leqslant\vert e_j\wedge \bar u\vert^2=\vert\bar u\vert^2-\langle e_j,\bar u\rangle^2,$$
    and $$\langle\hat e_j\circ u,(e_j\wedge\bar u)u\rangle=\langle e_j,(e_j\wedge \bar u)u\rangle=\langle\bar u,u\rangle-\langle e_j,u\rangle\langle e_j,\bar u\rangle,$$
    we see that
    \begin{equation}\label{eq: upper bound of X_j}
        \begin{aligned}
            \int_{\Sp^n}\vert X_j\vert^2\leqslant 1-\int_{\Sp^n}\langle e_j,u\rangle^2\dif\sigma-(\vert\bar u\vert^2-\langle e_j,\bar u\rangle^2) =1-\vert\bar u\vert^2-\int_{\Sp^n}\langle e_j,u-\bar u\rangle^2\dif\sigma.
        \end{aligned}
    \end{equation}
    Combining \eqref{eq: calculation of second variation on conformal vector fields}, \eqref{eq: upper bound of X_j} and Proposition \ref{prop: stable inequality}, we see
    $$\int_{\Sp^n}3\langle e_j,u\rangle^2\vert\nabla u\vert^2\dif\sigma-E(u)\geqslant\frac{(n-1)^2}{4}\left(\int_{\Sp^n}\langle e_j,u-\bar u\rangle^2\dif\sigma-(1-\vert\bar u\vert^2)\right).$$
    Summing over $j$ gives the desired estimate.
\end{proof}

\subsection{Jacobi operator and index of harmonic maps}\label{ss: pre, jacobi operator and index}
For simplicity, we restrict the discussion of the Jacobi operator and the index to the harmonic maps $u\in C^{\infty}(\Sp^n,\Sp^k)$. Given the bilinear form $I_u$, there is one linear operator associated to it.

\begin{defn}
    Let $u\in C^{\infty}(\Sp^n,\Sp^k)$.
    \begin{enumerate}
        \item The \textbf{Jacobi operator} $L_u$ associated to $u$ is characterized by
        $$I_u(X,Y)=-\int_{\Sp^n}\langle L_u X,Y\rangle\dif\sigma,\text{ for all }X,Y\in W^{1,2}(\Sp^n,T_u\Sp^k).$$
        Alternatively, the Jacobi operator can be written explicitly by
        $$L_uX=\Delta X+\vert\nabla u\vert^2X+2\langle\nabla u,\nabla X\rangle u,\quad\text{ for }X\in C^{\infty}(\Sp^n,T_u\Sp^k).$$
        \item If a section $X$ solves $L_uX=0$, $X$ is called a \textbf{Jacobi field} of $u$.
        \item The \textbf{index} of $u$ is defined to be the dimension of maximal subspace of $W^{1,2}(\Sp^n,T_u\Sp^k)$ on which $I_u$ is negative definite. Alternatively, the index of $u$ equals to the number of negative eigenvalues of $L_u$, counted with multiplicity. 
    \end{enumerate}
\end{defn}

Jacobi fields should be viewed as infinitesimal variations of harmonic maps. More specifically, if $(u_t)_{t\in(-\epsilon,\epsilon)}$ is a family of harmonic maps, depending smoothly on $t$, then $\p_tu_t$ is a Jacobi field for all $t\in(-\epsilon,\epsilon)$.

We need the following proposition regarding the Jacobi field and the index.
\begin{prop}[Y. L. Xin \cite{Xin} and A. El Soufi \cite{ElSoufiIndex}]\label{prop: fundamental space of index}
    Let $u\in C^{\infty}(\Sp^n,\Sp^k)$ be a non-constant harmonic map.
    \begin{enumerate}
        \item For any $v\in\R^{n+1}$, the section $\nabla_{\hat v}u\in C^{\infty}(\Sp^n,T_u\Sp^k)$ is an eigensection of $L_u$, corresponding to the eigenvalue $-(n-2)$. Recall that $\hat v(\omega)=v-\langle\omega,v\rangle\omega\in T_{\omega}\Sp^n$ is the conformal vector field on $\Sp^n$.;
        \item Let $\lbrace e_1,\dots,e_{n+1}\rbrace$ be an orthonormal frame of $\R^{n+1}$. Then $\lbrace\nabla_{\hat e_{\alpha}}u\rbrace_{\alpha=1}^{n+1}$ forms an $(n+1)$-dimensional space. In particular, the index of $u$ is at least $n+1$, provided $n\geqslant3$.
    \end{enumerate}
\end{prop}
\begin{proof}
    For the completeness of the paper, we record a simple proof of this fact.
    
    For the part (1) of the lemma, let us consider the homogeneous extension $\tilde u(x)=u(x/\vert x\vert)$. Then for all $\omega\in\Sp^n$, using the fact that $\tilde u$ does not have radial derivative, we have 
    $$\nabla^{\R^{n+1}}_v\tilde u(\omega)=\nabla^{\Sp^n}_{\Pi_{T_{\omega}\Sp^n}v}u(\omega)=\nabla^{\Sp^n}_{\hat v(\omega)}u(\omega).$$
    Since $\tilde u(\cdot+tv)$ is harmonic for all $t$, we see $\nabla^{\R^{n+1}}_v\tilde u=\p_t\tilde u(\cdot+tv)|_{t=0}$ is a Jacobi field of $\tilde u$. Let us write the Jacobi field equation of $\nabla^{\R^{n+1}}_v\tilde u$ in polar coordinates $x=r\omega,r>0,\omega\in\Sp^n$. Note that $\nabla^{\R^{n+1}}_v\tilde u(x)=\nabla^{\R^{n+1}}_v\tilde u(\omega)/r=\nabla^{\Sp^n}_{\hat v}u(\omega)/r$.
    \begin{align*}
        0=L_{\tilde u}^{\R^{n+1}}\nabla_{v}^{\R^{n+1}}\tilde u&=\left(\frac{\p^2}{\p r^2}+\frac{n}{r}\frac{\p}{\p r}+\frac{1}{r^2}L_u^{\Sp^n}\right)\frac{\nabla^{\Sp^n}_{\hat v}u(\omega)}{r}\\&=(-(n-2)+L_u^{\Sp^n})\frac{\nabla^{\Sp^n}_{\hat v}u(\omega)}{r^3}.
    \end{align*}
    Taking $r=1$, we get (1) of the lemma. For (2), we need to show that $\lbrace\nabla_{\hat e_{\alpha}}^{\Sp^n} u\rbrace_{1\leqslant\alpha\leqslant n+1}$ is linearly independent. To see this, assume there exist $a_1,\dots, a_{n+1}$, such that
    $$0=\sum_{\alpha=1}^{n+1}a_{\alpha}\nabla^{\Sp^{n+1}}_{\hat{e}_{\alpha}}u(\omega)=\nabla^{\Sp^{n+1}}_{\hat{v}}u(\omega)=\nabla^{\R^{n+1}}_v\tilde u(\omega),$$
    where $v=\sum_{\alpha=1}^{n+1}a_{\alpha}e_{\alpha}$. As a result, $\tilde u$ is translation invariant in direction $v$. If $v\ne0$, the singular set of $\tilde u$ would contain $\R v$, contradicting the smoothness of $u$. Hence $v$ must be $0$, which implies the conclusion of the lemma. 
\end{proof}

\subsection{Spherical harmonic components}\label{ss: spherical harmonic}

In this subsection we recall the basic facts about the eigenvalues and eigenfunctions of the Laplace $\Delta$ operator on round spheres. 

The eigenvalues of $-\Delta$ on $\Sp^n$ are $\lambda_l=l(l+n-1)$ for $l=0,1,2\dots$ In particular, the first two eigenvalues and the corresponding eigenfunctions are
$$\lambda_0=0,\text{ eigenfunction }1;$$
$$\lambda_1=n,\text{ eigenfunctions }x_{\alpha},\alpha=1,\dots,n+1,$$
where the $x_{\alpha}$ are coordinate functions on $\R^{n+1}$, restricted to $\Sp^n$. And the higher eigenfunctions are given by the restriction of homogeneous harmonic polynomials to $\Sp^n$, correspond to eigenvalues bounded from below by $2(n+1)$. Moreover, they satisfy the normalization condition using normalized volume measure $\sigma$:
$$\int_{\Sp^n}1\dif\sigma=1;\int_{\Sp^n}x_{\alpha}\dif\sigma=0,\int_{\Sp^n}x_{\alpha}x_{\beta}\dif\sigma=\frac{1}{n+1}\delta_{\alpha\beta}\text{ for all }\alpha,\beta.$$
By spectral theory, any function $f\in W^{1,2}(\Sp^n)$ admits an $L^2$-orthogonal decomposition
$$f(x)=\bar f+\sum_{\alpha=1}^{n+1}a_{\alpha}x_{\alpha}+h(x),\text{ for }\bar f=\int_{\Sp^n}f\dif \sigma,a_\alpha=(n+1)\int_{\Sp^n}fx_{\alpha}\dif\sigma,$$
and the higher order term $h$ satisfies the estimate
$$2(n+1)\int_{\Sp^n}h^2\dif\sigma\leqslant\int_{\Sp^n}\vert\nabla h\vert^2\dif\sigma.$$

Given a harmonic map $u:\Sp^n\rightarrow\Sp^k$, to quantify the extent to which its image occupies different directions in the target $\Sp^k$, we consider its projection to the first eigenfunctions: 
$$a_{\alpha}=(n+1)\int_{\Sp^n}ux_{\alpha}\dif\sigma\in\R^{k+1}.$$
These vectors give rise to a $(k+1)\times(n+1)$ matrix (all vectors are regarded as column vectors).
\begin{defn}
    Let $u\in C^\infty(\Sp^n,\Sp^k)$ and vectors $a_{\alpha}$ be as above. Define
    $$\mathscr{H}_u=(a_1,\dots,a_{n+1}),$$
    and call it the \textbf{first harmonic coefficient matrix} of $u$.
\end{defn}
Heuristically, the dimension of $\operatorname{Im}\mathscr{H}_u$ should be thought as the virtual dimension of the image of $u$. Let us expand $u$ in terms of spherical harmonic functions
$$u(x)=\bar u+\mathscr{H}_ux+\text{higher order terms},\text{ where }\bar u=\int_{\Sp^n}u\dif\sigma\in\R^{k+1}.$$
The following fact is important in the remainder part of this paper. For any vector $e\in(\operatorname{Im}\mathscr{H}_u)^{\perp}$, we have that 
$$0=\langle a_{\alpha},e\rangle=(n+1)\int_{\Sp^n}\langle u,e\rangle x_{\alpha}\dif\sigma\text{ for all }\alpha.$$
This is equivalent to saying that the function $\langle u,e\rangle$ does not have first order term in its spherical harmonic decomposition. Given the lower bound of the eigenvalues $\lambda_l\geqslant2(n+1)$ for $l\geqslant2$, we have that
\begin{equation}\label{eq: better estimate for higher order terms}
    2(n+1)\int_{\Sp^n}\langle u-\bar u,e\rangle^2\dif\sigma\leqslant\int_{\Sp^n}\vert\nabla\langle u,e\rangle\vert^2\dif\sigma.
\end{equation}
We will see in later sections that how this improvement give rise to a better lower bound of the energy. 

\section{Nonexistence of stable tangent maps, when \texorpdfstring{$(n,k)\neq(5,6)$}{(n,k)!=(5,6)}}\label{s: part 1}
In this section, we are going to prove most of the Theorem \ref{thm: main theorem 2}.
\begin{thm}\label{thm: main theorem 2 part 1}
    Let $u\in C^{\infty}(\Sp^n,\Sp^k)$ be a harmonic map whose homogeneous extension $u(x/\vert x\vert)$ is stable. If $n<k$, $2\leqslant n\leqslant5$ and $(n,k)\ne(5,6)$, then $u$ is a constant.    
\end{thm}
\begin{proof}
    Recall that the map $u$ induces its first harmonic coefficient matrix $\mathscr{H}_u$, which is a $(k+1)\times (n+1)$ matrix. Let us take a linear subspace $L\subset(\operatorname{Im}\mathscr{H}_u)^{\perp}$. Since the rank of $\mathscr{H}_u$ is at most $n+1<k+1$, we can take $d=\dim L$ to be at least 1. By the discussion in Subsection \ref{ss: spherical harmonic}, in particular \eqref{eq: better estimate for higher order terms}, we have two inequalities related to $L$.
    $$2(n+1)\int_{\Sp^n}\vert\Pi_L(u-\bar u)\vert^2\dif\sigma\leqslant\int_{\Sp^n}\vert\nabla\Pi_Lu\vert^2\dif\sigma=\int_{\Sp^n}\vert\nabla u\vert^2\vert\Pi_Lu\vert^2,$$
    where we used the equation \eqref{eq: energy equality of components} in the last equality. The $L^{\perp}$ component instead satisfies
    \begin{equation}\label{eq: upper bound of L-energy}
        E(u)-\int_{\Sp^n}\vert\nabla\Pi_L u\vert^2\dif\sigma=\int_{\Sp^n}\vert\nabla\Pi_{L^{\perp}} u\vert^2\dif\sigma\geqslant n\int_{\Sp^n}\vert\Pi_{L^{\perp}}(u-\bar u)\vert^2\dif\sigma.
    \end{equation}
    Combining these two inequalities, we get that
    \begin{equation}\label{eq: energy lower bound wrt kernel}
    \begin{aligned}
        E(u)&\geqslant n\int_{\Sp^n}\vert\Pi_{L^{\perp}}(u-\bar u)\vert^2\dif\sigma+2(n+1)\int_{\Sp^n}\vert\Pi_{L}(u-\bar u)\vert^2\dif\sigma\\&=n(1-\vert\bar u\vert^2)+(n+2)\int_{\Sp^n}\vert\Pi_L(u-\bar u)\vert^2\dif\sigma.
    \end{aligned}
    \end{equation}
    On the other hand, putting \eqref{eq: upper bound of L-energy} into Proposition \ref{prop: control wrt L}, we get
    \begin{equation}\label{eq: improved upper bound wrt L}
        \begin{aligned}
            0\leqslant&3\int_{\Sp^n}\vert\nabla u\vert^2\vert\Pi_Lu\vert^2\dif\sigma-dE(u)+\frac{(n-1)^2}{4}\left(d(1-\vert\bar u\vert^2)-\int_{\Sp^n}\vert\Pi_L(u-\bar u)\vert^2\dif\sigma\right)\\\leqslant&(3-d)E(u)+\left(\frac{(n-1)^2}{4}d-3n\right)(1-\vert\bar u\vert^2)+\left(3n-\frac{(n-1)^2}{4}\right)\int_{\Sp^n}\vert\Pi_L(u-\bar u)\vert^2\dif\sigma.
        \end{aligned}
    \end{equation}
    Note that the coefficient $3n-(n-1)^2/4$ is positive when $2\leqslant n\leqslant 5$. Canceling the term $\int_{\Sp^n}\vert\Pi_L(u-\bar u)\vert^2\dif\sigma$ in \eqref{eq: energy lower bound wrt kernel} and \eqref{eq: improved upper bound wrt L} yields
    \begin{equation}\label{eq: canceling L-energy}
        \begin{aligned}
            \left(\frac{(6n+6)n}{n+2}-\frac{(n-1)^2}{4}\left(d+\frac{n}{n+2}\right)\right)(1-\vert\bar u\vert^2)\leqslant\left(\frac{-n^2+14n-1}{4(n+2)}+(3-d)\right)E(u).
        \end{aligned}
    \end{equation}
    Let us compare \eqref{eq: canceling L-energy} with the last inequality of Proposition \ref{prop: control wrt L}, which is 
    \begin{equation}\label{eq: full energy control}
        E(u)\leqslant\frac{(n-1)^2k}{4(k-2)}(1-\vert\bar u\vert^2).
    \end{equation}
    
    We have two cases.
    
    \begin{itemize}[leftmargin=*]
        \item $2\leqslant n\leqslant 4$.
        
        In this case, let us simply take $d=1$. Then \eqref{eq: canceling L-energy} becomes
        $$E(u)\geqslant(2n+2)\left(\frac{-n^2+14n-1}{-n^2+22n+15}\right)(1-\vert\bar u\vert^2).$$
        Let us look at the numbers $n$ which fulfills the inequality $$(2n+2)\left(\frac{-n^2+14n-1}{-n^2+22n+15}\right)\geqslant\frac{(n-1)^2k}{4(k-2)}.$$
        When $k\geqslant n+1$, the right hand side is at most $(n^2-1)/4$, which takes value $3/4,2,15/4$ when $n=2,3,4$ respectively, which are less than the corresponding left hand side value $138/55,32/9,130/29$ (when $n=5$ and $k=6$, LHS$=132/25<6=$RHS, though).

        This upper bound of $E(u)$ in \eqref{eq: full energy control} is even less than the lower bound of $E(u)$ above for $2\leqslant n\leqslant4$ and $k\geqslant n+1$, unless $E(u)=0$. Hence, $u$ is forced to be constant for $n,k$ in the current range. 
        \item $n=5$ and $k\geqslant7$.
        In this case, $(\operatorname{Im}\mathscr{H}_u)^{\perp}$ has dimension at least $k-n\geqslant2$. Hence, we may select $d=2$ in the preceeding discussion. Given $n=5$, \eqref{eq: canceling L-energy} and \eqref{eq: full energy control} respectively become
        $$\frac{52}{9}(1-\vert\bar u\vert^2)\leqslant E(u)\text{ and }E(u)\leqslant\frac{4k}{k-2}(1-\vert\bar u\vert^2)\leqslant\frac{28}{5}(1-\vert\bar u\vert^2).$$
        Since $52/9>28/5$, again $u$ is forced to be a constant.
    \end{itemize}
\end{proof}

\section{Nonexistence of stable tangent maps, when \texorpdfstring{$(n,k)=(5,6)$}{(n,k)=(5,6)}}\label{s: part 2}

Now, let us focus on the remaining part of Theorem \ref{thm: main theorem 2}.
\begin{thm}\label{thm: main theorem part 2}
    Let $u\in C^{\infty}(\Sp^5,\Sp^6)$ be a harmonic map whose homogeneous extension $u(x/\vert x\vert)$ is stable. Then $u$ is a constant.
\end{thm}
\begin{proof}
    Throughout this proof let us assume that $u$ is not a constant to get a contradiction. Under this assumption, we have that $E(u)>0$ and $\vert\bar u\vert<1$. Let $L$ be a $d$-dimensional subspace of $\R^7$. Having fixed $n=5$ and $k=6$, conclusions of Proposition \ref{prop: control wrt L} take special forms
    \begin{equation}\label{eq: control wrt L, n=5}
        3\int_{\Sp^5}\vert\nabla u\vert^2\vert\Pi_Lu\vert^2\dif\sigma-dE(u)\geqslant4\left(\int_{\Sp^5}\vert\Pi_L(u-\bar u)\vert^2\dif\sigma-d(1-\vert\bar u\vert^2\right),
    \end{equation}
    together with taking $d=1$ in \eqref{eq: canceling L-energy}, we have two-sided estimate of $E(u)$, 
    \begin{equation}\label{eq: full energy control, n=5}
        \frac{132}{25}(1-\vert\bar u\vert^2)\leqslant E(u)\leqslant6(1-\vert\bar u\vert^2).
    \end{equation}
    
    We need the following estimate regarding the eigenvalues of $\mathscr{H}_u\mathscr{H}_u^{\text{T}}$, which is a $7\times 7$ matrix.
    
    \begin{lem}\label{lem: eigenvalue lower bound}
        Let $0\leqslant\lambda_1\leqslant\cdots\leqslant\lambda_7$ be the eigenvalues of $\mathscr{H}_u\mathscr{H}_u^{\text{T}}$. We have
        $$\left(\sum_{j=1}^7\sqrt{\lambda_j}\right)^2\geqslant\frac{600}{143}E(u).$$
    \end{lem}
    \begin{proof}
    Since the eigenvalues of the Laplace operator on $\Sp^5$ are $0,5,12\dots$, the Poincare inequality applied to $\Pi_{L^{\perp}}(u-\bar u)$ gives
    \begin{equation}\label{eq: lemma 4.2 eq1}
        E(u)-\int_{\Sp^5}\vert\nabla\Pi_Lu\vert^2\dif\sigma\geqslant5\left((1-\vert\bar u\vert^2-\int_{\Sp^5}\vert\Pi_L(u-\bar u)\vert^2\dif\sigma\right).
    \end{equation}
    
    If we write $u=\bar u+\mathscr{H}_ux+$higher order terms, we see $\Pi_Lu=\Pi_L\bar u+\Pi_{L}\mathscr{H}_ux+$higher order terms is still of the form of decomposition through eigenfunctions of $\Sp^5$. Hence,
    \begin{equation}\label{eq: lem 4.2 eq2}
        \int_{\Sp^5}\vert\nabla\Pi_L u\vert^2\dif\sigma\geqslant5\int_{\Sp^5}\vert\Pi_L\mathscr{H}_ux\vert^2\dif\sigma+12\int_{\Sp^5}(\vert\Pi_L(u-\bar u)\vert^2-\vert\Pi_L\mathscr{\mathscr{H}_u}x\vert^2)\dif\sigma.
    \end{equation}
    Eliminating the terms $\int_{\Sp^5}\vert\nabla \Pi_Lu\vert^2\dif\sigma$ and $\int_{\Sp^5}\vert\Pi_L(u-\bar u)\vert^2\dif\sigma$ in above two inequalities and \eqref{eq: control wrt L, n=5}, we first deduce (using \eqref{eq: control wrt L, n=5} and \eqref{eq: lemma 4.2 eq1})
    $$\int_{\Sp^5}\vert\Pi_L(u-\bar u)\vert^2\dif\sigma\geqslant\frac{(d-3)E(u)+(15-4d)(1-\vert\bar u\vert^2)}{11},$$
    and then (combine \eqref{eq: lemma 4.2 eq1},\eqref{eq: lem 4.2 eq2})
    $$\int_{\Sp^5}\vert\Pi_L\mathscr{H}_ux\vert^2\dif\sigma\geqslant\frac{(7d-32)E(u)+(160-28d)(1-\vert\bar u\vert^2)}{77}.$$
    To finish the proof of lemma, we notice that the normalization $\int_{\Sp^5}\vert x_{\alpha}\vert^2\dif\sigma=1/6$ implies
    $$\int_{\Sp^5}\vert\Pi_L\mathscr{H}_ux\vert^2\dif\sigma=\operatorname{tr}(\Pi_L\mathscr{H}_u\mathscr{H}_u^{\text{T}})/6.$$ By taking $L$ to be the subspace of $\R^7$ spanned by the eigenvectors associated to $\lambda_1,\dots,\lambda_d$, together with the lower bound of $\int_{\Sp^5}\vert\Pi_L\mathscr{H}_u\vert^2\dif\sigma$, we see
    \begin{equation}\label{eq: lower bound of the sum of eigenvalues}
        \Lambda_d:=\sum_{j=1}^d\lambda_j\geqslant\frac{6}{77}((7d-32)E(u)+(160-28d)(1-\vert\bar u\vert^2).
    \end{equation}
    Let us set 
    $$\frac{1-\vert\bar u\vert^2}{E(u)}=\frac{1}{6}+\delta\text{ and }\kappa=\frac{E(u)}{77}.$$
    Given \eqref{eq: full energy control, n=5}, we see $0\leqslant\delta\leqslant1/44$. Since $\mathscr{H}_u$ has rank at most $6$, we see that $\Lambda_1=\lambda_1=0$. The lower bound for all other $d$ is now written as
    $$\Lambda_2\geqslant\kappa(-4+624\delta)_+=:\kappa p(\delta);\quad\Lambda_d\geqslant\kappa(14d-32+(960-168d))\delta\text{ for }3\leqslant d\leqslant7.$$
    As such, we see $\Lambda=(\lambda_2,\dots,\lambda_7)$ weak supermajorization $$a(\delta)=\kappa(p(\delta),10+456\delta-p(\delta),14-168\delta,\dots,14-168\delta).$$
    We refer readers to Appendix \ref{ap: weakly supermajorize} for the definition and consequence of weakly supermajorize. One can check that $a(\delta)$ is arranged in an increasing order so that it satisfies the definiton in Appendix \ref{ap: weakly supermajorize}. Since the function $t\mapsto\sqrt{t}$ is concave and increasing, we have
    $$\sum_{j=1}^7\sqrt{\lambda_j}\geqslant\sqrt{\kappa}\left(\sqrt{p(\delta)}+\sqrt{10+456\delta-p(\delta)}+4\sqrt{14-168\delta}\right)=:\sqrt\kappa F(\delta),$$
    where,
    $$F(\delta)=\begin{cases}
        \sqrt{10+456\delta}+4\sqrt{14-168\delta},&\text{ if }0\leqslant\delta\leqslant1/156;\\
        \sqrt{-4+624\delta}+5\sqrt{14-168\delta},&\text{ if }1/156\leqslant\delta\leqslant1/44.
    \end{cases}$$
    Each expression above is a concave function. Hence, to obtain a lower bound of $F$, we only need the endpoint values of $F$. They are $F(0)=\sqrt{10}+4\sqrt{14},F(1/156)=10\sqrt{42/13},F(1/44)=6\sqrt{112/11}$. The minimum number here is $F(1/156)=10\sqrt{42/13}$. Hence,
    $$\left(\sum_{j=1}^7\sqrt{\lambda_j}\right)^2\geqslant\kappa\frac{4200}{13}=\frac{600E(u)}{143}.$$
    \end{proof}
    Next, let us derive an upper bound of $\sum_j\sqrt{\lambda_j}$. Note that $\{\sqrt{\lambda_j}\}$ are singular values of $\mathscr{H}_u$.
    \begin{lem}\label{lem: eigenvalue upper bound}
        We  have the upper bound
        $$\left(\sum_{j=1}^7\sqrt{\lambda_j}\right)^2\leqslant\frac{144}{25}E(u)(6(1-\vert\bar u\vert^2)-E(u)).$$
    \end{lem}
    \begin{proof}
        To prove this lemma, we need to make use of the stability inequality Proposition \ref{prop: stable inequality}. For simplicity, set $H(X,Y)=I_u(X,Y)+4\int_{\Sp^n}\langle X,Y\rangle\dif\sigma$ for $X,Y\in C^{\infty}(\Sp^5,T_u\Sp^6)$. Then, $H(X,X)\geqslant0$. Recall also that Proposition \ref{prop: fundamental space of index} tells us for any unit vector $e\in\R^6$, $\nabla_{\hat e}u$ is an eigensection of $L_u$ corresponding to eigenvalue $-3$.

        Let $\lbrace e_{\alpha}\rbrace_{\alpha=1}^6$ and $\lbrace e_j\rbrace_{j=7}^{13}$ be orthonormal bases of $\R^6$ and $\R^7$ respectively, and consider the vector fields $X_{\alpha}=\nabla_{\hat{e}_{\alpha}}u,1\leqslant\alpha\leqslant6$ and $Y_j=\hat e_j\circ u-(e_j\wedge \bar u)u,7\leqslant j\leqslant 13$. We calculate using eigensection equation of $X_{\alpha}$
        $$H(X_{\alpha},X_{\alpha})=I_u(X_{\alpha},X_{\alpha})+4\int_{\Sp^5}\vert X_\alpha\vert^2\dif\sigma=\int_{\Sp^5}\vert X_{\alpha}\vert^2\dif\sigma,$$
        $$\sum_{\alpha=1}^6H(X_{\alpha},X_{\alpha})=\sum_{\alpha=1}^6\int_{\Sp^5}\vert\nabla_{\hat e_{\alpha}}u\vert^2\dif\sigma= E(u).$$
        Using Proposition \ref{prop: infinitesimal isometry},
        $$\int_{\Sp^5}\langle(e_j\wedge\bar u)u,\nabla_{\hat e_{\alpha}}u\rangle\dif\sigma=-\frac{1}{3}I_u((e_j\wedge \bar u)u,\nabla_{\hat e_{\alpha}}u\rangle=0.$$
        Hence, 
        \begin{align*}
            H(X_{\alpha},Y_j)&=\int_{\Sp^5}\langle\nabla_{\hat e_{\alpha}}u,\hat e_j\circ u\rangle\dif\sigma=\int_{\Sp^5}\langle\nabla_{\hat e_{\alpha}}u,e_j\rangle\dif\sigma\\&=\int_{\Sp^5}\langle\nabla\langle u,e_j\rangle,\nabla x_{\alpha}\rangle\dif\sigma=5\int_{\Sp^5}\langle u,e_j\rangle x_{\alpha}\dif\sigma,
        \end{align*}
        where in the last step we used the fact that $x_{\alpha}$ is the eigenfunction for eigenvalue $5$. By the definition of $\mathscr{H}_u$, the matrix $(H(Y_j,X_{\alpha}))$ equals $5\mathscr{H}_u/6$. Finally, as computed in \eqref{eq: calculation of second variation on conformal vector fields} and \eqref{eq: upper bound of X_j}, we have that
        \begin{align*}
            \sum_{j=1}^7H(Y_j,Y_j)&=\sum_{j=1}^7\left(I_u(\hat e_j\circ u,\hat e_j\circ u)+4\int_{\Sp^5}\vert \hat e_j\circ u-(e_j\wedge\bar u)u\vert^2\dif\sigma\right)\\&\leqslant4(6(1-\vert\bar u\vert^2)-E(u)).
        \end{align*}
        
        Set matrices $\mathscr{X}=(H(X_{\alpha},X_{\beta})),\mathscr{Y}=(H(Y_i,Y_j))$. Since $H$ is a positive semidefinite bilinear form, we have that the following matrix is positive semidefinite:
        $$
        \begin{pmatrix}
            \mathscr{X}&\frac{5}{6}\mathscr{H}^{\text{T}}\\
            \frac{5}{6}\mathscr{H}&\mathscr{Y}
        \end{pmatrix}
        $$
        Hence we may write this matrix as a product $\mathscr{A}^{\text{T}}\mathscr{A}$ for a $13\times 13$ matrix $\mathscr{A}$. We further write $\mathscr{A}=(\mathscr{B},\mathscr{C})$, where $\mathscr{B}$ and $\mathscr{C}$ are $13\times 6$ and $13\times 7$ matrices respectively. As a result, we can write $\mathscr{X}=\mathscr{B}^{\text{T}}\mathscr{B},\mathscr{Y}=\mathscr{C}^{\text{T}}\mathscr{C}$ and $5\mathscr{H}/6=\mathscr{C}^{\text{T}}\mathscr{B}$.
         
        Let $\Vert \cdot\Vert_*$ be the nuclear norm of matrices, which is just the sum of singular values. We deduce that 
        $$\left(\sum_{j=1}^7\frac{5}{6}\sqrt{\lambda_j}\right)^2=\left\Vert\frac{5}{6}\mathscr{H}_u\right\Vert_*^2\leqslant\operatorname{tr}(\mathscr{B}^{\text{T}}\mathscr{B})\operatorname{tr}(\mathscr{C}^{\text{T}}\mathscr{C})=\operatorname{tr}\mathscr{X}\operatorname{tr}\mathscr{Y}\leqslant 4E(u)(6(1-\vert\bar u\vert^2)-E(u)).$$
    \end{proof}
    Now, let us derive a contradiction. Combining Lemma \ref{lem: eigenvalue lower bound} and Lemma \ref{lem: eigenvalue upper bound}, we get
    $$6(1-\vert\bar u\vert^2)-E(u)\geqslant\frac{600}{143}\frac{25}{144}=\frac{625}{858}.$$
    On the other hand, the lower bound in \eqref{eq: full energy control, n=5} means
    $$6(1-\vert\bar u\vert^2)-E(u)\leqslant\frac{3}{22}E(u).$$Hence,
    $$6-(6(1-\vert\bar u\vert^2)-E(u))\geqslant E(u)\geqslant\frac{22}{3}(6(1-\vert\bar u\vert^2)-E(u)).$$
    We deduce that $$6(1-\vert\bar u\vert^2)-E(u)\leqslant\frac{18}{25}.$$
    But $625/858>18/25$, this is a contradiction.
\end{proof}

\section{Index estimates and rigidity results}\label{s: index estimate}

In this section, we derive the index estimate in Theorem \ref{thm: main theorem 3}. 

Recall that Proposition \ref{prop: fundamental space of index} already gives us $n+1$ negative directions of $u$, and the corresponding eigenvalue is $-(n-2)$. To find an extra negative eigenvalue, and hence prove Theorem \ref{thm: main theorem 3}, we only need to show that the first eigenvalue of $L_u$ is less than $-(n-2)$. We establish this by contradiction.

\begin{thm}\label{thm: eq-index estimate}
    Let $u\in C^{\infty}(\Sp^n,\Sp^k)$ be a harmonic map. Suppose $6\leqslant n\leqslant k-1$, and the first eigenvalue of $L_u$ is at least $-(n-2)$. Then $u$ must be a constant.
\end{thm}
\begin{proof}
    Let $L$ be a $d$-dimensional subspace of $\R^{k+1}$. Similar to Proposition \ref{prop: control wrt L}, we have that
    $$3\int_{\Sp^n}\vert\nabla u\vert^2\vert\Pi_Lu\vert^2\dif\sigma-dE(u)\geqslant(n-2)\left(\int_{\Sp^n}\vert\Pi_L(u-\bar u)\vert^2\dif\sigma-d(1-\vert\bar u\vert^2\right).$$
    In particular,
    $$E(u)\leqslant\frac{(n-2)k}{k-2}(1-\vert\bar u\vert^2).$$
    On the other hand, take $L$ to be a 1-dimensional subspace of $(\operatorname{Im}\mathscr{H}_u)^{\perp}$. Similar to the derivation of \eqref{eq: canceling L-energy}, we have
    $$\frac{2(n+1)^2}{2n+3}(1-\vert\bar u\vert^2)\leqslant E(u).$$
    
    Now, in the case where $k\geqslant n+1,$
    $$\frac{(n-2)k}{k-2}-\frac{2(n+1)^2}{2n+3}\leqslant\frac{(n-2)(n+1)}{n-1}-\frac{2(n+1)^2}{2n+3}=-\frac{(n+4)(n+1)}{(2n+3)(n-1)}<0.$$
    We must have $E(u)=0$ as stated.
\end{proof}

\appendix
\section{Weak supermajorization}\label{ap: weakly supermajorize}
We refer readers to \cite{MarshallOlkinArnold} for a detailed discussion for this concept.
\begin{defn}
    Given two vectors $a=(a_1,\dots,a_n)$ and $b=(b_1,\dots,b_n)$ arranged so that $a_i\leqslant a_{i+1}$ and $b_i\leqslant b_{i+1}$ for each $i\in\lbrace1,\dots,n\rbrace$. We say $a$ \textbf{weakly supermajorize} $b$ if
    $$\sum_{i=1}^lb_i\geqslant\sum_{i=1}^la_i,\text{ for each }l\in\lbrace1,\dots,n\rbrace.$$
\end{defn}
The property regrading weakly supermajorization used in the proof of Lemma \ref{lem: eigenvalue lower bound} is the following.
\begin{prop}[{\cite[4.B.2]{MarshallOlkinArnold}}]
    Given two vectors $a=(a_1,\dots,a_n)$ and $b=(b_1,\dots,b_n)$ arranged in an increasing order. Then $b$ weakly supermajorize $a$ if and only if for each increasing concave function $\phi$, there holds
    $$\sum_{i=1}^n\phi(b_i)\geqslant\sum_{i=1}^n\phi(a_i).$$
\end{prop}
\bibliography{reference}
\bibliographystyle{plain}
\end{document}